\documentclass[11]{article}
\usepackage{amsmath, amssymb, graphicx, amsthm, mathrsfs}
\usepackage{mathptmx}
\usepackage[italic]{mathastext}
\usepackage{mathtools} % For better math formatting
\usepackage[utf8]{inputenc}
\usepackage{enumitem}  % For customizing itemize environment
\usepackage[colorlinks=true]{hyperref}
\usepackage{geometry}
\newtheorem{definition}{Definition}[section]
\newtheorem{theorem}[definition]{Theorem}
\newtheorem{lemma}[definition]{Lemma}
\newtheorem{remark}[definition]{Remark}
\newtheorem{example}[definition]{Example}

\newtheorem{corollary}[definition]{Corollary}

\newtheorem{note}[definition]{Note}
\numberwithin{equation}{section}

\title{Common Fixed Points via Rational $\varphi$-Contractive Conditions in Orbitally Complete Metric Spaces}

\author{Babu G.V.R\textsuperscript{1}, Alemayehu G. Negash\textsuperscript{2,*}, M. L. Sandhya \textsuperscript{3},  Meaza F. Bogale\textsuperscript{4}\\
\small \textsuperscript{1,3}Department of Mathematics, Andhra University, India\\
\small \textsuperscript{1}Email: \texttt{gvr\_babu@hotmail.com} \\
\small \textsuperscript{3}Email: sandhya\_mudunur@yahoo.co.in\\
\small \textsuperscript{2,4}Department of Mathematics, Hampton University, USA\\
$^*$ \textit{Corresponding author.} \\
\small \textsuperscript{2}Email: \texttt{alemayehu.negash@hamptonu.edu}\\
\small \textsuperscript{4}Email: \texttt{meaza.fantahun@hamptonu.edu}
}

\date{}

\begin{document}

\maketitle

\begin{abstract}
This paper establishes the existence and uniqueness of common fixed points for three selfmaps $A$, $S$, and $T$ defined on an orbitally complete metric space. The maps are assumed to satisfy a generalized $\varphi$-contractive condition involving rational expressions. The analysis is carried out under mild continuity assumptions—namely, reciprocal continuity and compatibility (or compatibility of type (A)) between pairs of mappings. The results not only extend earlier fixed point theorems by Jaggi and Phaneendra et al., but also unify several classical results by employing a unified orbital framework. Additional corollaries and illustrative examples are provided, along with a convergence theorem for sequences of selfmaps.


\textbf{AMS(2000) Mathematics Subject Classification:} 47H10, 54H25.

\textbf{Corresponding author:} Alemayehu G. Negash
\end{abstract}

\section{INTRODUCTION}
Fixed point theory is a central area of nonlinear analysis with wide-ranging applications in various fields including differential equations, dynamic programming, and mathematical economics. A cornerstone result in this theory is the Banach Contraction Principle, which guarantees the existence and uniqueness of fixed points for contractive selfmaps on complete metric spaces.

Over the years, this foundational principle has been extended and generalized in many directions. One significant line of development concerns the study of common fixed points for multiple selfmaps. Results in this area are not only theoretically interesting but also essential in iterative methods and computational fixed point approximations.

In 1975, Jaggi\cite{Jaggi1975} introduced a framework for analyzing common fixed points of two selfmaps under contractive conditions in complete metric spaces. 

\begin{theorem}[Jaggi \cite{Jaggi1975}, Theorem 4, page. 227] \label{thm:1.1}
Let \(S\) and \(T\) be two selfmaps defined on a complete metric space \((X,d)\) satisfying the following conditions:
\begin{enumerate}
    \item[(i)] for some \(\alpha,\beta\in[0,1)\) with \(\alpha+\beta<1\),
    \begin{equation}\label{eq:1.1.1}
        d(Sx,Ty)\leq\alpha \frac{d(x,Sx)}{d(x,y)}\frac{d(y,Ty)}{d(x,y)}+\beta d(x,y), \text{ for all }x,y\in X, x\neq y;
    \end{equation}
    \item[(ii)] \(ST\) is continuous on \(X\); and
    \item[(iii)] there exists an \(x_{0}\in X\) such that in the sequence \(\{x_{n}\}\), where
    \[
        x_{n} = \begin{cases}
            Sx_{n-1}, & \text{when } n \text{ is odd}, \\
            Tx_{n-1}, & \text{when } n \text{ is even},
        \end{cases}
    \]
    \(x_{n}\neq x_{n+1}\) for all \(n\). Then \(S\) and \(T\) have a unique common fixed point.
\end{enumerate}
\end{theorem}

Later, Phaneendra and Swatmaram refined this approach by incorporating concepts like orbital completeness and weak commutativity, which allow for a more flexible treatment of mappings beyond the classical contractive setting.

\begin{theorem}[\cite{Phaneendra2007}, Theorem 2, page. 25] \label{thm:1.5}
Let \(A\), \(S\) and \(T\) be continuous selfmaps of \(X\) satisfying the following inequality: there exist nonnegative reals \(\alpha\), \(\beta\) and \(\gamma\) with \(\alpha+\beta<1\) and \(\alpha+\gamma<1\) such that
\begin{equation}\label{eq:1.5.1}
d(Sx,Ty)\leq\alpha d(Ax,Ay)+\beta \frac{d(Ax,Sx) d(Ay,Ty)}{d(Ax, Ay)}+\gamma \frac{d(Ax,Ty) d(Ay,Sx)}{d(Ax, Ay)}
\end{equation}
if \(Ax\neq Ay\) and \(d(Sx,Ty)=0\) if \(d(Ax,Ay)=0\) for all \(x,y\in X\).

Suppose there exists an orbit \(\mathcal{O}_{STA}(x_{0})\) of some \(x_{0}\in X\), given by \eqref{eq:1.2.1}. Then the sequence \(\{Ax_{n}\}_{n=1}^{\infty}\) is a Cauchy sequence in the orbit \(\mathcal{O}_{STA}(x_{0})\) of \(x_{0}\in X\).

Further, if \(X\) is orbitally complete at \(x_{0}\in X\) and \((A,S)\) and \((A,T)\) are weakly commuting pairs, then \(A\), \(S\) and \(T\) have a unique common fixed point.
\end{theorem}

Motivated by these developments, we study a more general scenario involving three selfmaps $A$, $S$, and $T$ on a metric space. We impose a generalized contractive condition, known as the $\varphi$-contractive inequality, which subsumes a wide class of rational-type inequalities used in earlier literature. Additionally, we relax completeness requirements by considering orbital completeness and allow for weaker forms of continuity and compatibility.

Our aim is to establish the existence and uniqueness of a common fixed point under a set of generalized assumptions involving rational $\varphi$-contractive conditions. In addition, we extend our results to sequences of selfmaps and investigate the convergence behavior of their fixed point sequences.The results in this paper not only generalize but also unify several known fixed point theorems, providing a broader theoretical foundation for further study and application.

\section{Preliminaries}

\begin{definition}[\cite{Phaneendra2007}]
Let \(A\), \(S\) and \(T\) be three selfmaps of a metric space \((X,d)\) and \(x_{0}\in X\). If there exists a sequence \(\{x_{n}\}_{n=1}^{\infty}\subset X\) such that
\begin{equation}\label{eq:1.2.1}
Ax_{n} = \begin{cases}
Sx_{n-1}, & \text{if } n \text{ is odd}, \\
Tx_{n-1}, & \text{if } n \text{ is even},
\end{cases}
\end{equation}
for \(n=1,2,3,\ldots\), then \(\{Ax_{1},Ax_{2},\ldots\}\) is said to \textit{constitute \((S,T,A)\)-orbit or \((S,T,A)\)-orbit of \(x_{0}\in X\)}. We denote it by \(\mathcal{O}_{STA}(x_{0})\) and its closure is denoted by \(\overline{\mathcal{O}_{STA}(x_{0})}\).
\end{definition}

\begin{definition}[\cite{Phaneendra2007}]
A metric space \((X,d)\) is said to be \((S,T,A)\)-\textit{orbitally complete} at \(x_{0}\in X\) if every Cauchy sequence in \(\mathcal{O}_{STA}(x_{0})\) converges in \(X\).
\end{definition}

\begin{definition}[\cite{Sessa1982}] \label{ex:1.6}
Two selfmaps \(A\) and \(S\) of a metric space \((X,d)\) are said to be \textit{weakly commuting} on \(X\), if \(d(SAx,ASx)\leq d(Ax,Sx)\) for all \(x\in X\). We denote it by 'the pair \((A,S)\) is weakly commuting on \(X\)'.
\end{definition}

\begin{example}
Let \(X=\{0,1\}\) with the usual metric. Define \(A:X\to X\) by \(A0=1\), \(A1=0\). Define \(T:X\to X\) by \(Tx=x\) for all \(x\in X\). Suppose that \(S=A\).

Here \(\mathcal{O}_{STA}(0)=\{0,1,0,0,1,1,\ldots\}\) and \(\mathcal{O}_{STA}(1)=\{1,0,1,1,0,0,\ldots\}\). We observe that \(\mathcal{O}_{STA}(0)\) and \(\mathcal{O}_{STA}(1)\) have two limit points \(0\) and \(1\). The pairs of mappings \((A,S)\) and \((A,T)\) are weakly commuting pairs and the mappings
$A$, $S$ and $T$ satisfy inequality \eqref{eq:1.5.1} with any $\alpha\geq 0$, $\beta\geq 0$ and $\gamma\geq 0$ with $\alpha+\beta<1$ and $\alpha+\gamma<1$ so that $A$, $S$ and $T$ satisfy all the hypotheses of Theorem \ref{thm:1.5}. Observe that $A$, $S$ and $T$ do not have a common fixed point.
\end{example}

\textbf{Observation:} $Ax_{n}=Ax_{n+1}$ for infinitely many $n$ in Example \ref{ex:1.6}.

Given a sequence of selfmaps $\{S_{n}\}_{n=1}^{\infty}$ on a metric space $(X,d)$ with $\{S_{n}\}_{n=1}^{\infty}$ converges pointwise to $S$ on $X$ and $S_{n}u_{n}=u_{n}$ for all $n=1,2,3,\ldots$ with $\{u_{n}\}_{n=1}^{\infty}$ converges to $u$ as $n$ tends to infinity, it is natural to ask 'whether this $u$ is a fixed point of $S$'. Its converse is also interesting; i.e., if $Su=u$, then 'Is $\{u_{n}\}_{n=1}^{\infty}$ converges to $u$ as $n$ tends to infinity?' Related results were obtained by \cite{Bonsall1962}, \cite{Nadlar1969}, \cite{Fraser1969}, \cite{Furi1969}. In 1975, D.S. Jaggi \cite{Jaggi1975} proved the following theorem.

\begin{theorem}[\cite{Jaggi1975}, Theorem 5, P. 288]
Let $\{S_{n}\}_{n=1}^{\infty}$ be a sequence of selfmaps defined on a metric space $(X,d)$ with $S_{n}u_{n}=u_{n}$, $n=1,2,\ldots$. Assume that
\begin{enumerate}
    \item[(i)] there exists $\alpha,\beta\in[0,1)$ with $\alpha+\beta<1$ such that
    \begin{equation}\label{eq:1.7.1}
        d(S_{n}x,S_{n}y)\leq\alpha \frac{d(x,S_{n}x)}{d(x,y)}\frac{d(y,S_{n}y)}{d(x,y)}+\beta d(x,y)
    \end{equation}
    for all $x,y\in X$, $x\neq y$, $n=1,2,3,\ldots$ and
    \item[(ii)] $\{S_{n}\}$ converges pointwise to $S$. Then $u_{n}\to u$ if and only if $u$ is a fixed point of $S$.
\end{enumerate}
\end{theorem}

\begin{definition}[\cite{Jungck1986}]
Two selfmaps $A$ and $S$ of a metric space $(X,d)$ are said to be \textit{compatible} if $\lim\limits_{n\to\infty}d(SAx_{n},ASx_{n})=0$ whenever $\{x_{n}\}$ is a sequence in $X$ such that $\lim\limits_{n\to\infty}Ax_{n}=\lim\limits_{n\to\infty}Sx_{n}=t$ for some $t$ in $X$. We denote it by 'the pair $(A,S)$ is compatible'.
\end{definition}

We note that every weakly commuting pair of selfmaps is compatible but its converse need not be true \cite{Jungck1986}.

\begin{definition}[\cite{Jungck1993}]
Two selfmaps $A$ and $S$ of a metric space $(X,d)$ are said to be \textit{compatible of type (A)} if $\lim\limits_{n\to\infty}d(SAx_{n},S^{2}x_{n})=0$ and $\lim\limits_{n\to\infty}d(SAx_{n},A^{2}x_{n})=0$ whenever $\{x_{n}\}$ is a sequence in $X$ such that $\lim\limits_{n\to\infty}Ax_{n}=\lim\limits_{n\to\infty}Sx_{n}=t$ for some $t$ in $X$. We denote it by 'the pair $(A,S)$ is compatible of type (A)'.
\end{definition}

Examples are given to show that the two concepts of compatibility are independent \cite{Jungck1993}.

\begin{lemma}
If $A$ and $S$ are either compatible maps or compatible maps of type (A), then they commute at their coincidence point \cite{Jungck1986,Jungck1993}.
\end{lemma}

%\section*{Definitions and Lemmas}

\begin{definition}[\cite{Pant1998}]
Two mappings $A$ and $S$ of a metric space $(X,d)$ are called \textit{reciprocal continuous} on $X$ if, $\lim\limits_{n\to\infty}ASx_{n}=At$ and $\lim\limits_{n\to\infty}SAx_{n}=St$ whenever $\{x_{n}\}$ is a sequence in $X$ such that $\lim\limits_{n\to\infty}Ax_{n}=\lim\limits_{n\to\infty}Sx_{n}=t$ for some $t$ in $X$. We denote it by 'the pair $(A,S)$ is reciprocal continuous on $X$'.
\end{definition}

We observe that if $A$ and $S$ are continuous then they are reciprocal continuous. But its converse need not be true \cite{Pant1998}.

Let $\Phi$ be the set of all continuous selfmaps $\varphi:\mathbb{R}_{+}\to \mathbb{R}_{+}$ satisfying
\begin{enumerate}
    \item[(i)] $\varphi$ is monotonically increasing; and
    \item[(ii)] $0\leq\varphi(t)<t$ for all $t>0$.
\end{enumerate}

\begin{lemma}[\cite{Singh1977}]
For any $t\in(0,\infty)$, if $\varphi(t)<t$, then $\lim_{n\to\infty}\varphi^{n}(t)=0$, where $\varphi^{n}$ denotes the $n$-times repeated composition of $\varphi$ with itself.
\end{lemma}

\begin{definition} \label{def:1.3}
Let $(X,d)$ be a metric space and $S,T,A:X\to X$. If there exists a $\varphi\in\Phi$ such that
\begin{equation}\label{eq:1.13.1}
d(Sx,Ty)\leq\varphi\left(\max\left\{d(Ax,Ay),\frac{d(Ax,Sx)}{d(Ax,Ay)}, \frac{d(Ax,Ty)}{d(Ax,Ay)}\right\}\right)
\end{equation}
if $Ax\neq Ay$ and $d(Sx,Ty)=0$ if $d(Ax,Ay)=0$ for all $x,y\in X$, then we call the mappings $S$, $T$ and $A$ satisfy '$\varphi$-\textit{contractive inequality involving rational expressions}'.
\end{definition}

In Definition \ref{def:1.3}, if $S=T$, then the inequality \eqref{eq:1.13.1} becomes
\begin{equation}\label{eq:1.13.2}
d(Sx,Sy)\leq\varphi\left(\max\left\{d(Ax,Ay),\frac{d(Ax,Sx)}{d(Ax,Ay)}, \frac{d(Ax,Sy)}{d(Ax,Ay)}\right\}\right)
\end{equation}
if $Ax\neq Ay$ and $d(Sx,Sy)=0$ if $d(Ax,Ay)=0$ for all $x,y\in X$, then we call the mappings $S$ and $A$ satisfy '$\varphi$-\textit{contractive inequality involving rational expressions}'. In addition, if $A$ is the identity mapping on $X$, then the inequality \eqref{eq:1.13.1} becomes
\begin{equation}\label{eq:1.13.3}
d(Sx,Sy)\leq\varphi\left(\max\left\{d(x,y),\frac{d(x,Sx)}{d(x,y)}, \frac{d(x,Sy)}{d(x,y)}\right\}\right)
\end{equation}
if $x\neq y$ and $d(Sx,Sy)=0$ if $d(x,y)=0$ for all $x,y\in X$, then we call the mapping $S$ satisfies '$\varphi$-\textit{contractive inequality involving rational expressions}'.

We denote the set of all positive integers by $\mathbb{N}$ and $\mathbb{R}_{+}=[0,\infty)$.

In this paper we prove the existence of common fixed points for three selfmaps $A$, $S$ and $T$ defined on an orbitally complete metric space under the
Under the assumption that (i) $A$, $S$ and $T$ satisfy $\varphi$-contractive inequality (Definition \ref{def:1.3}, (ii) the pairs $(A,S)$ and $(A,T)$ are reciprocal continuous and (iii) the pairs $(A,S)$ and $(A,T)$ are either compatible or compatible of type $(A)$. Also it is extended to a sequence of selfmaps. In section 4, we write some corollaries and provide examples in support of our main results. In section 5, we prove a result on the convergence of sequence of common fixed points. Our main results generalize and modify the results of \cite{Jaggi1975} and \cite{Phaneendra2007}.

\section{MAIN RESULTS}
To prove our main results, we need the following lemma which shows that the sequence $\{Ax_{n}\}$ defined by \eqref{eq:1.2.1} is Cauchy.

\begin{lemma}\label{lem:2.1}
Let $(X,d)$ be a metric space and let $S,T,A:X\to X$. Suppose that there exists a $\varphi\in\Phi$ such that the selfmappings $S$, $T$ and $A$ satisfy $\varphi$-contractive inequality \eqref{eq:1.13.1}. Suppose also that there exists an orbit $\mathcal{O}_{STA}(x_{0})$ of some $x_{0}\in X$ given by \eqref{eq:1.2.1}. Then the sequence $\{Ax_{n}\}$ is a Cauchy sequence in the orbit $\mathcal{O}_{STA}(x_{0})$.
\end{lemma}

\begin{proof}
Suppose that the sequence $\{Ax_{n}\}$ is given by \eqref{eq:1.2.1}.

If $Ax_{2n-1}=Ax_{2n}$ for some $n$, then $d(Ax_{2n},Ax_{2n-1})=0$ and hence $d(Sx_{2n},Tx_{2n-1})=0$. This implies that $Sx_{2n}=Tx_{2n-1}$, i.e., $Ax_{2n+1}=Ax_{2n}$. Then $d(Ax_{2n},Ax_{2n+1})=0$ and hence $d(Sx_{2n},Tx_{2n+1})=0$. This implies that $Sx_{2n}=Tx_{2n+1}$, i.e., $Ax_{2n+1}=Ax_{2n+2}$. Hence inductively we get $Ax_{2n}=Ax_{2n+k}$ for $k=1,2,3,\ldots$. Hence $\{Ax_{m}\}_{m\geq n}$ is a constant sequence and hence Cauchy.

Now, assume that $Ax_{n}\neq Ax_{n+1}$ for all $n$. Without loss of generality we assume $n$ is even. Taking $x=x_{n+1}$ and $y=x_{n+2}$ in \eqref{eq:1.13.1}, we get
\begin{align*}
d(Ax_{n+1},Ax_{n+2}) &= d(Sx_{n},Tx_{n+1}) \\
&\leq \varphi\left(\max\left\{d(Ax_{n},Ax_{n+1}),\frac{d(Ax_{n},Sx_{n})}{d(Ax_{n},Ax_{n+1})}Tx_{n+1}),\right.\right. \\
&\quad \left.\left.\frac{d(Ax_{n},Tx_{n+1})}{d(Ax_{n},Ax_{n+1})}\frac{d(Ax_{n+1},Sx_{n})}{d(Ax_{n},Ax_{n+1})}\right\}\right) \\
&= \varphi\left(\max\left\{d(Ax_{n},Ax_{n+1}),\frac{d(Ax_{n},Ax_{n+1})}{d(Ax_{n},Ax_{n+1})}\frac{d(Ax_{n+1},Ax_{n+2})}{d(Ax_{n},Ax_{n+1})},\right.\right. \\
&\quad \left.\left.\frac{d(Ax_{n},Ax_{n+2})}{d(Ax_{n},Ax_{n+1})}\frac{d(Ax_{n+1},Ax_{n+1})}{d(Ax_{n},Ax_{n+1})}\right\}\right) \\
&= \varphi\left(\max\left\{d(Ax_{n},Ax_{n+1}),d(Ax_{n+1},Ax_{n+2})\right\}\right).
\end{align*}

If $\max\{d(Ax_{n},Ax_{n+1}),d(Ax_{n+1},Ax_{n+2})\}=d(Ax_{n+1},Ax_{n+2})$, then we get
\begin{equation}\label{eq:2.1.1}
d(Ax_{n+1},Ax_{n+2}) \leq \varphi(d(Ax_{n},Ax_{n+1})).
\end{equation}

But again from \eqref{eq:1.13.1} with $x=x_{n}$ and $y=x_{n+1}$, it follows that
\begin{align*}
d(Ax_{n},Ax_{n+1}) &= d(Ax_{n+1},Ax_{n}) = d(Sx_{n},Tx_{n-1}) \\
&\leq \varphi\left(\max\left\{d(Ax_{n},Ax_{n-1}), \frac{d(Ax_{n},Sx_{n})}{d(Ax_{n},Ax_{n-1})}d(x_{n-1},Tx_{n-1}),\right.\right. \\
&\quad \left.\left.\frac{d(Ax_{n},Tx_{n-1})}{d(Ax_{n},Ax_{n-1})}d(x_{n-1},Sx_{n})\right\}\right) \\
&= \varphi\left(\max\left\{d(Ax_{n},Ax_{n-1}), \frac{d(Ax_{n},Ax_{n+1})}{d(Ax_{n},Ax_{n-1})}d(x_{n-1},Ax_{n}),\right.\right. \\
&\quad \left.\left.\frac{d(Ax_{n},Ax_{n})}{d(Ax_{n},Ax_{n-1})}d(x_{n-1},Ax_{n+1})\right\}\right) \\
&= \varphi\left(\max\left\{d(Ax_{n},Ax_{n+1}), d(Ax_{n-1},Ax_{n})\right\}\right) \\
&= \varphi(d(Ax_{n-1},Ax_{n})).
\end{align*}

Hence,

\begin{equation}\label{eq:2.1.2}
d(Ax_{n},Ax_{n+1}) \leq \varphi(d(Ax_{n-1},Ax_{n})) \text{ for } n=1,2,3,\ldots.
\end{equation}

Since $\varphi(t)<t$ for $t>0$, from \eqref{eq:2.1.2}, we have 
\begin{equation}\label{eq:2.1.3}
    \{d(Ax_{n},Ax_{n+1})\}_{n=1}^{\infty}
\end{equation}

is a decreasing sequence of reals.
Thus, from repeated application of \eqref{eq:2.1.2} and monotone increasing property of $\varphi$, we get

\begin{equation}\label{eq:2.1.4}
d(Ax_{n+1},Ax_{n+2}) \leq \varphi^{n}(d(Ax_{1},Ax_{2})), \quad n=1,2,3,\ldots.
\end{equation}

Letting $n\to\infty$ in \eqref{eq:2.1.4}, by Lemma 1.12, we get

\begin{equation}\label{eq:2.1.5}
d(Ax_{n+1},Ax_{n+2}) \to 0.
\end{equation}

Thus, from \eqref{eq:2.1.3}, \eqref{eq:2.1.4} and \eqref{eq:2.1.5}, to show that $\{Ax_{n}\}$ is Cauchy, it is sufficient to show $\{Ax_{2n}\}$ is Cauchy. Otherwise, there exists an $\varepsilon>0$ and there exist sequences $\{m_{k}\}$ and $\{n_{k}\}$ with $m_{k}>n_{k}>k$ such that

\begin{equation}\label{eq:2.1.6}
d(Ax_{2m_{k}},Ax_{2n_{k}}) \geq \varepsilon \qquad \text{and} \qquad d(Ax_{2m_{k}-2},Ax_{2n_{k}}) < \varepsilon.
\end{equation}

Now for each positive integer $k$,

\begin{align*}
\varepsilon &\leq d(Ax_{2m_{k}},Ax_{2n_{k}}) \\
&\leq d(Ax_{2m_{k}},Ax_{2m_{k}-1}) + d(Ax_{2m_{k}-1},Ax_{2m_{k}-2}) \\
&\quad + d(Ax_{2m_{k}-2},Ax_{2n_{k}}).
\end{align*}

On taking limits as $k\to\infty$, and using \eqref{eq:2.1.5} and \eqref{eq:2.1.6}, we have

\begin{equation}\label{eq:2.1.7}
\lim_{n\to\infty} d(Ax_{2m_{k}},Ax_{2n_{k}}) = \varepsilon.
\end{equation}

Now, from the triangle inequality, we have

\[
|d(Ax_{2m_{k}},Ax_{2n_{k}-1}) - d(Ax_{2m_{k}},Ax_{2n_{k}})| \leq d(Ax_{2n_{k}},Ax_{2n_{k}-1});
\]

On taking limits as $k\to\infty$, and using \eqref{eq:2.1.5} and \eqref{eq:2.1.7}, we have

\begin{equation}\label{eq:2.1.8}
\lim_{k\to\infty} d(Ax_{2m_{k}},Ax_{2n_{k}-1}) = \varepsilon.
\end{equation}

Again from the triangle inequality, we have

\begin{align*}
|d(Ax_{2m_{k}+1},Ax_{2n_{k}-1}) - d(Ax_{2m_{k}},Ax_{2n_{k}})| &\leq d(Ax_{2n_{k}-1},Ax_{2n_{k}}) \\
&\quad + d(Ax_{2m_{k}+1},Ax_{2m_{k}}).
\end{align*}

On taking limits as $k\to\infty$, and using \eqref{eq:2.1.5} and \eqref{eq:2.1.7}, we have

\begin{equation}\label{eq:2.1.9}
\lim_{k\to\infty} d(Ax_{2m_{k}+1},Ax_{2n_{k}-1}) = \varepsilon.
\end{equation}

Now

\begin{align*}
d(Ax_{2m_{k}},Ax_{2n_{k}}) &\leq d(Ax_{2m_{k}},Ax_{2m_{k}+1}) + d(Ax_{2m_{k}+1},Ax_{2n_{k}}) \\
&= d(Ax_{2m_{k}},Ax_{2m_{k}+1}) + d(Sx_{2m_{k}},Tx_{2n_{k}-1}) \\
&\leq d(Ax_{2m_{k}},Ax_{2m_{k}+1}) \\
&\quad + \varphi\left(\max\left\{d(Ax_{2m_{k}},Ax_{2n_{k}-1}),\right.\right. \\
&\quad \left.\left.\frac{d(Ax_{2m_{k}},Sx_{2m_{k}})}{d(Ax_{2m_{k}},Ax_{2n_{k}-1})}, \frac{d(Ax_{2m_{k}},Tx_{2n_{k}-1})}{d(Ax_{2m_{k}},Ax_{2n_{k}-1})} \frac{d(Ax_{2n_{k}-1},Sx_{2m_{k}})}{d(Ax_{2m_{k}},Ax_{2n_{k}-1})}\right\}\right) \\
&= d(Ax_{2m_{k}},Ax_{2m_{k}+1}) \\
&\quad + \varphi\left(\max\left\{d(Ax_{2m_{k}},Ax_{2n_{k}-1}), \frac{d(Ax_{2m_{k}},Ax_{2m_{k}+1})}{d(Ax_{2m_{k}},Ax_{2n_{k}-1})},\right.\right. \\
&\quad \left.\left.\frac{d(Ax_{2m_{k}},Ax_{2n_{k}})}{d(Ax_{2m_{k}},Ax_{2n_{k}-1})} \frac{d(Ax_{2n_{k}-1},Ax_{2m_{k}+1})}{d(Ax_{2m_{k}},Ax_{2n_{k}-1})}\right\}\right).
\end{align*}

Letting $n\to\infty$, using the continuity of $\varphi$, and using \eqref{eq:2.1.5}, \eqref{eq:2.1.7}, \eqref{eq:2.1.8} and \eqref{eq:2.1.9} we get

\[
\varepsilon \leq 0 + \varphi\left(\max\left\{\varepsilon, \tfrac{0}{\varepsilon}, \tfrac{\varepsilon^{2}}{\varepsilon}\right\}\right) = \varphi(\varepsilon),
\]

a contradiction. Thus $\{Ax_{2n}\}$ is Cauchy and hence $\{Ax_{n}\}$ is a Cauchy sequence. Thus, Lemma \ref{lem:2.1} follows.
\end{proof}

\begin{theorem}\label{thm:2.2}
Let $(X,d)$ be a metric space and let $S,T,A:X\to X$. Suppose that there exists a $\varphi\in\Phi$ such that the selfmappings $S$, $T$ and $A$ satisfy $\varphi$-contractive inequality \eqref{eq:1.13.1}. Suppose also that there exists an orbit $\mathcal{O}_{STA}(x_{0})$ of some $x_{0}\in X$ given by \eqref{eq:1.2.1}. Further assume that
\begin{enumerate}
    \item[(3.2.1)] $X$ is orbitally complete at $x_{0}\in X$;
    \item[(3.2.2)] the pairs $(A,S)$ and $(A,S)$ are reciprocal continuous.
\end{enumerate}
If the pairs $(A,S)$ and $(A,T)$ are compatible, then $S$, $T$ and $A$ have a unique common fixed point.
\end{theorem}

\begin{proof}
Suppose that sequence $\{Ax_{n}\}$ is given by \eqref{eq:1.2.1}. Then by Lemma \ref{lem:2.1} the sequence $\{Ax_{n}\}$ is Cauchy.

If $Ax_{2n}=Ax_{2n+1}$ for some $n$, then $d(Ax_{2n},Ax_{2n+1})=0$ and hence $d(Sx_{2n},Tx_{2n+1})=0$. This implies that $Sx_{2n}=Tx_{2n+1}$, i.e., $Ax_{2n+1}=Ax_{2n+2}$. Now $Ax_{2n}=Ax_{2n+1}$ for some $n$ implies that $Ax_{2n}=Sx_{2n}$ for some $n$; and with a similar argument, we get $Ax_{2n+1}=Ax_{2n+2}=Tx_{2n+1}$.

Let $w=Au=Su$ and $z=Av=Tv$, where $u=x_{2n}$ and $v=x_{2n+1}$. Since $(A,S)$ and $(A,T)$ are compatible they commute at their coincidence point. So, $Sw=SAu=ASu=Aw$ and $Tz=TAv=ATv=Az$. Hence

\begin{equation}
Sw=Aw \quad \text{and} \quad Tz=Az.
\end{equation}

As $d(Au,Av)=d(Ax_{2n},Ax_{2n+1})=0$, we get $d(Su,Tv)=0$ and hence $Su=Tv$, i.e., $w=Tv$. So,

\begin{equation}
Aw=ATv=TAv=Tw, \quad \text{i.e.,} \quad Aw=Tw.
\end{equation}

Hence, $Sw=Aw=Tw$.

Now we claim $Aw=A(Aw)$. If $Aw\neq A(Aw)$, then

\begin{align*}
d(Aw,A(Aw)) &= d(Aw,A(Tw))=d(Sw,T(Aw)) \\
&\leq \varphi\left(\max\left\{d(Aw,A(Aw)), \frac{d(Aw,Sw)}{d(Aw,A(Aw))}T(Aw),\right.\right. \\
&\quad \left.\left.\frac{d(Aw,T(Aw))}{d(Aw,A(Aw))}\right\}\right) \\
&= \varphi(d(Aw,A(Aw))), \quad \text{a contradiction}.
\end{align*}

Hence, $A(Aw)=Aw$.

Therefore, we have:

\begin{align*}
S(Aw) &= A(Sw)=A(Aw)=Aw \quad \text{and} \\
T(Aw) &= A(Tw)=A(Aw)=Aw.
\end{align*}

Thus, $Aw$ is a common fixed point of $S$, $T$ and $A$.

Now suppose $Ax_{n}\neq Ax_{n+1}$ for all $n$. Since the sequence $\{Ax_{n}\}$ is Cauchy in $X$, by (3.2.1) there exists $z\in X$ such that

\begin{equation}
\lim_{n\to\infty}Ax_{n}=z.
\end{equation}

and hence,

\begin{equation}\label{eq:2.2.3}
\lim_{n\to\infty}Ax_{2n+1}=\lim_{n\to\infty}Sx_{2n}=\lim_{n\to\infty}Ax_{2n+2}=\lim_{n\to\infty}Tx_{2n+1}=u.
\end{equation}

Now since $(A,S)$ is reciprocal continuous, we have

\begin{equation}\label{eq:2.2.4}
\lim_{n\to\infty}SAx_{2n}=Su \quad \text{and} \quad \lim_{n\to\infty}ASx_{2n}=Au.
\end{equation}

Again, since $(A,T)$ is reciprocal continuous, we have

\begin{equation}\label{eq:2.2.5}
\lim_{n\to\infty}TAx_{2n+1}=Tu \quad \text{and} \quad \lim_{n\to\infty}ATx_{2n+1}=Au.
\end{equation}

Now since $(A,S)$ are compatible,

\begin{align*}
d(ASx_{2n},Su) &\leq d(ASx_{2n},SAx_{2n})+d(SAx_{2n},Su).
\end{align*}

Letting $n\to\infty$, using \eqref{eq:2.2.4} we get

From the previous steps, we have:

\begin{equation}\label{eq:2.2.6}
\lim_{n\to\infty}ASx_{2n}=Su.
\end{equation}

Hence,

\begin{equation}\label{eq:2.2.7}
Au=Su.
\end{equation}

Similarly, since $(A,T)$ are compatible,

\begin{align*}
d(ATx_{2n+1},Tu) &\leq d(ATx_{2n+1},TAx_{2n+1})+d(TAx_{2n+1},Tu).
\end{align*}

Letting $n\to\infty$, using \eqref{eq:2.2.5} we get

\begin{equation}
\lim_{n\to\infty}d(ATx_{2n+1},Tu)=0,
\end{equation}

which implies

\begin{equation}\label{eq:2.2.8}
\lim_{n\to\infty}ATx_{2n+1}=Tu.
\end{equation}

Hence,

\begin{equation}\label{eq:2.2.9}
Au=Tu.
\end{equation}

From \eqref{eq:2.2.7} and \eqref{eq:2.2.9}, it follows that

\begin{equation}\label{eq:2.2.10}
Au=Su=Tu.
\end{equation}

Now we claim that $Su=u$. If $Su\neq u$, then $d(Su,u)>0$. Since $d(Au,Ax_{2n+1})\to d(Au,u)=d(Su,u)>0$ as $n\to\infty$, we have $d(Au,Ax_{2n+1})>0$ for large $n$. Therefore, for large $n$, we have

\begin{align*}
d(Su,Tx_{2n+1}) &\leq \varphi\left(\max\left\{d(Au,Ax_{2n+1}), \frac{d(Au,Su)}{d(Au,Ax_{2n+1})},\right.\right. \\
&\quad \left.\left.\frac{d(Au,Tx_{2n+1})}{d(Au,Ax_{2n+1})}\frac{d(Ax_{2n+1},Su)}{d(Ax_{2n+1})}\right\}\right).
\end{align*}

Letting $n\to\infty$, using \eqref{eq:2.2.3} and \eqref{eq:2.2.10} we get

\begin{align*}
d(Su,u) &\leq \varphi\left(\max\left\{d(Su,u),\frac{0}{d(Su,u))},\frac{d(u,Su)}{d(Su,u)}\right\}\right) \\
&= \varphi(d(Su,u)), \quad \text{a contradiction}.
\end{align*}

Hence, $d(Su,u)=0$ which implies $Su=u$. Therefore,

\begin{equation}
Au=Tu=Su=u.
\end{equation}

Uniqueness follows from the inequality \eqref{eq:1.13.1}. Hence the result follows.

%\subsection*{Extension of Theorem 2.2}

In Theorem \ref{thm:2.2}, we replace the condition 'the pair $(A,S)$ and $(A,T)$ are compatible' by 'the pair $(A,S)$ and $(A,T)$ are compatible of type (A)' and hence obtain the following theorem.
\end{proof}

\begin{theorem}\label{thm:2.3}
Let $(X,d)$ be a metric space and let $S,T,A:X\to X$. Suppose that there exists a $\varphi\in\Phi$ such that the selfmappings $S$, $T$ and $A$ satisfy $\varphi$-contractive inequality \eqref{eq:1.13.1}. Suppose also that there exists an orbit $\mathcal{O}_{STA}(x_{0})$ of some $x_{0}\in X$ given by \eqref{eq:1.2.1}. Further assume that
\begin{enumerate}
    \item[(3.3.1)] $X$ is orbitally complete at $x_{0}\in X$;
    \item[(3.3.2)] the pairs $(S,A)$ and $(T,A)$ are reciprocal continuous.
\end{enumerate}
If the pairs $(A,S)$ and $(A,T)$ are compatible of type (A), then $S$, $T$ and $A$ have a unique common fixed point.
\end{theorem}

\begin{proof}
Suppose that sequence $\{Ax_{n}\}$ is given by \eqref{eq:1.2.1}. Then by Lemma \ref{lem:2.1} the sequence $\{Ax_{n}\}$ is Cauchy. If $Ax_{2n}=Ax_{2n+1}$ for some $n$. Then $d(Ax_{2n},Ax_{2n+1})=0$ and hence $d(Sx_{2n},Tx_{2n+1})=0$. This implies that $Sx_{2n}=Tx_{2n+1}$, i.e., $Ax_{2n+1}=Ax_{2n+2}$. Now $Ax_{2n}=Ax_{2n+1}$ for some $n$ implies that $Ax_{2n}=Sx_{2n}$ for some $n$; and with a similar argument, we get $Ax_{2n+1}=Ax_{2n+2}=Tx_{2n+1}$.

Let $w=Au=Su$ and $z=Av=Tv$, where $u=x_{2n}$ and $v=x_{2n+1}$. Since $(A,S)$ and $(A,T)$ are compatible of type (A) they commute at their coincidence point. So, $Sw=SAu=ASu=Aw$ and $Tz=TAv=ATv=Az$. Hence
\begin{equation}
Sw=Aw \quad \text{and} \quad Tz=Az.
\end{equation}

As $d(Au,Av)=d(Ax_{2n},Ax_{2n+1})=0$, we get $d(Su,Tv)=0$ and hence $Su=Tv$, i.e., $w=Tv$. So,
\begin{equation}
Aw=ATv=TAv=Tw, \quad \text{i.e.,} \quad Aw=Tw.
\end{equation}

Hence, $Sw=Aw=Tw$.

Now we claim $Aw=A(Aw)$. If $Aw\neq A(Aw)$, then
\begin{align*}
d(Aw,A(Aw)) &= d(Aw,A(Tw))=d(Sw,T(Aw)) \\
&\leq \varphi\left(\max\left\{d(Aw,A(Aw)), \frac{d(Aw,Sw)}{d(Aw,A(Aw))}T(Aw)),\right.\right. \\
&\quad \left.\left.\frac{d(Aw,T(Aw))}{d(Aw,A(Aw))}\frac{d(Sw,A(Aw))}{d(Aw,A(Aw))}\right\}\right) \\
&= \varphi(d(Aw,A(Aw))), \quad \text{a contradiction}.
\end{align*}

Hence, $A(Aw)=Aw$.

Therefore, we have:
\begin{align*}
S(Aw) &= A(Sw)=A(Aw)=Aw \quad \text{and} \\
T(Aw) &= A(Tw)=A(Aw)=Aw.
\end{align*}

Therefore, $Aw$ is a common fixed point of $S$, $T$ and $A$.

Now suppose $Ax_{n}\neq Ax_{n+1}$ for all $n$. Since the sequence $\{Ax_{n}\}$ is Cauchy in $X$, by (3.3.1) there exists a $z\in X$ such that

\begin{equation}
\lim_{n\to\infty}Ax_{n}=u
\end{equation}

and hence,

\begin{equation}\label{eq:2.3.3}
\lim_{n\to\infty}Ax_{2n+1}=\lim_{n\to\infty}Sx_{2n}=\lim_{n\to\infty}Ax_{2n+2}=\lim_{n\to\infty}Tx_{2n+1}=u.
\end{equation}

Now since $(A,S)$ is reciprocal continuous, we have

\begin{equation}\label{eq:2.3.4}
\lim_{n\to\infty}SAx_{2n}=Su \quad \text{and} \quad \lim_{n\to\infty}ASx_{2n}=Au.
\end{equation}

As $S$ and $A$ are compatible of type $(A)$, we have

\begin{align*}
d(A^{2}x_{2n},Su) &\leq d(A^{2}x_{2n},SAx_{2n})+d(SAx_{2n},Su).
\end{align*}

Letting $n\to\infty$, and using \eqref{eq:2.3.3} and \eqref{eq:2.3.4}, we get

\begin{equation}\label{eq:2.3.5}
\lim_{n\to\infty}A^{2}x_{2n}=Su.
\end{equation}

and

\begin{align*}
d(S^{2}x_{2n},Au) &\leq d(S^{2}x_{2n},ASx_{2n})+d(ASx_{2n},Au).
\end{align*}

Letting $n\to\infty$, and using \eqref{eq:2.3.3} and \eqref{eq:2.3.4}, we get

\begin{equation}\label{eq:2.3.6}
\lim_{n\to\infty}S^{2}x_{2n}=Au.
\end{equation}

Again since $(A,T)$ is reciprocal continuous, we have

\begin{equation}\label{eq:2.3.7}
\lim_{n\to\infty}TAx_{2n+1}=Tu \quad \text{and} \quad \lim_{n\to\infty}ATx_{2n+1}=Au.
\end{equation}

and as $T$ and $A$ are compatible of type $(A)$, we have

\begin{align*}
d(A^{2}x_{2n+1},Tu) &\leq d(A^{2}x_{2n+1},TAx_{2n+1})+d(TAx_{2n+1},Tu).
\end{align*}

Letting $n\to\infty$, and using \eqref{eq:2.3.3} and \eqref{eq:2.3.7}, we get

\begin{equation}\label{eq:2.3.8}
\lim_{n\to\infty}A^{2}x_{2n+1}=Tu.
\end{equation}

and 

\begin{align*}
d(T^{2}x_{2n+1},Au) &\leq d(T^{2}x_{2n+1},ATx_{2n+1})+d(ATx_{2n+1},Au).
\end{align*}

Letting $n\to\infty$, and using \eqref{eq:2.3.3} and \eqref{eq:2.3.7}, we get

\begin{equation}\label{eq:2.3.9}
\lim_{n\to\infty}T^{2}x_{2n+1}=Au.
\end{equation}
Now we claim $Su=Tu$. If $Su\neq Tu$, then $d(Su,Tu)>0$. Since 
$d(A^{2}x_{2n},A^{2}x_{2n+1})\rightarrow d(Su,Tu)>0$ as 
$n\rightarrow \infty$, we have $d(A^{2}x_{2n},A^{2}x_{2n+1})>0$ for large $n$. Therefore, for large $n$, we have
\begin{align*}
d(SAx_{2n},TAx_{2n+1}) &\leq \varphi\left(\max\left\{d(A^{2}x_{2n},A^{2}x_{2n+1}),\right.\right. \\
&\quad \left.\left.\frac{d(A^{2}x_{2n},SAx_{2n})d(A^{2}x_{2n+1},TAx_{2n+1})}{d(A^{2}x_{2n},A^{2}x_{2n+1})},\right.\right. \\
&\quad \left.\left.\frac{d(A^{2}x_{2n},TAx_{2n+1})d(A^{2}x_{2n+1},SAx_{2n})}{d(A^{2}x_{2n},A^{2}x_{2n+1})}\right\}\right).
\end{align*}
Letting $n \to \infty$, and using equations \eqref{eq:2.3.3}, \eqref{eq:2.3.4}, and \eqref{eq:2.3.5}, we get

%Letting $n\rightarrow \infty$, using (2.3.4), (2.3.5), (2.3.7) and (2.3.6), we get
\begin{align*}
d(Su,Tu) &\leq \varphi\left(\max\left\{d(Su,Tu),\frac{0}{d(Su,Tu))}, \frac{d(Su,Tu)d(Su,Tu)}{d(Su,Tu)}\right\}\right) \\
&= \varphi(d(Su,Tu)), \text{ a contradiction.}
\end{align*}

Hence,
\begin{equation}\label{eq:2.3.10}
Su=Tu.
\end{equation}

Now we claim $Su=u$. If $Su\neq u$, then $d(Su,u)>0$. Since 
$d(A^{2}x_{2n},Ax_{2n+1})\rightarrow d(Su,u)>0$ as $n\rightarrow \infty$, we have 
$d(A^{2}x_{2n},Ax_{2n+1})>0$ for large $n$. Therefore, for large $n$, we have
\begin{align*}
d(SAx_{2n},Tx_{2n+1}) &\leq \varphi\left(\max\left\{d(A^{2}x_{2n},Ax_{2n+1}),\right.\right. \\
&\quad \left.\left.\frac{d(A^{2}x_{2n},SAx_{2n})d(Ax_{2n+1},Tx_{2n+1})}{d(A^{2}x_{2n},Ax_{2n+1})},\right.\right. \\
&\quad \left.\left.\frac{d(A^{2}x_{2n},Tx_{2n+1})d(Ax_{2n+1},SAx_{2n})}{d(A^{2}x_{2n},Ax_{2n+1})}\right\}\right).
\end{align*}

Letting $n\rightarrow \infty$, using \eqref{eq:2.3.3}, \eqref{eq:2.3.4} and \eqref{eq:2.3.5} we get
\begin{align*}
d(Su,u) &\leq \varphi\left(\max\left\{d(Su,u),\frac{0}{d(Su,u))}, \frac{d(u,Su)d(Su,u)}{d(Su,u)}\right\}\right) \\
&= \varphi(d(Su,u)), \text{ a contradiction.}
\end{align*}

Hence,
\begin{equation}\label{eq:2.3.11}
Su=u.
\end{equation}

Now, we claim $Au = Tu$. If $Au \neq Tu$, then $d(Au, Tu) > 0$. Since
\[
d(AS x_{2n}, A^2 x_{2n+1}) \to d(Au, Tu) > 0 \quad \text{as } n \to \infty,
\]
we have $d(AS x_{2n}, A^2 x_{2n+1}) > 0$ for large $n$. Therefore, for large $n$, we have
\begin{align}
    \label{eq:2.3.12-pre}
    d(S^2 x_{2n}, T A x_{2n+1}) 
    &\leq \varphi\left( 
        \max \left\{
        \frac{d(AS x_{2n}, A^2 x_{2n+1})}{d(AS x_{2n}, A^2 x_{2n+1})}, \right. \right. \notag\\
    &\quad \left. \left. 
        \frac{d(AS x_{2n}, S^2 x_{2n}) \cdot d(A^2 x_{2n+1}, T A x_{2n+1})}{d(AS x_{2n}, A^2 x_{2n+1})},
        \right. \right. \notag\\
    &\quad \left. \left. 
        \frac{d(AS x_{2n}, T A x_{2n+1}) \cdot d(S^2 x_{2n}, A^2 x_{2n+1})}{d(AS x_{2n}, A^2 x_{2n+1})}
        \right\}
    \right).
\end{align}

Letting $n \to \infty$, and using equations \eqref{eq:2.3.5}, \eqref{eq:2.3.6}, \eqref{eq:2.3.7}, and \eqref{eq:2.3.8}, we get
\begin{equation} \label{eq:2.3.12}
d(Au, Tu) \leq \varphi\left( \max \left\{
\frac{d(Au, Tu)}{d(Au, Tu)},
\frac{0 \cdot d(Au, Tu)}{d(Au, Tu)},
\frac{d(Au, Tu) \cdot d(Au, Tu)}{d(Au, Tu)}
\right\} \right) = \varphi(d(Au, Tu)),
\end{equation}
which is a contradiction. Hence, $Au = Tu$. \hfill $\text{(2.3.12)}$

From equations \eqref{eq:2.3.10}, \eqref{eq:2.3.11}, and \eqref{eq:2.3.12}, it follows that
\begin{equation} \label{eq:2.3.13}
Au = Tu = Su = u.
\end{equation}
Hence, $u$ is a common fixed point of $S$, $T$, and $A$.

The uniqueness of $u$ follows from inequality \eqref{eq:1.13.1}.
\end{proof}

\begin{remark}\label{rem:2.4}
Let $A$, $S$ and $T$ be selfmaps on a metric space $(X,d)$. If there exists an $(S,T,A)$-orbit $\mathcal{O}_{STA}(x_{0})$ for some $x_{0}\in X$ and selfmaps $A$, $S$ and $T$ satisfy all the conditions of either Theorem \ref{thm:2.2} or Theorem \ref{thm:2.3}, then $\{Ax_{n}\}_{n=1}^{\infty}\subset\mathcal{O}_{STA}(x_{0})$ is Cauchy, $Ax_{n}\to u$ as $n\to\infty$ and $u$ is the unique common fixed point of $A$, $S$ and $T$ and $u\in\mathcal{O}_{STA}(x_{0})$.
\end{remark}

In Theorems \ref{thm:2.2} and \ref{thm:2.3}, if $S=T$, then we have the following corollary.

\begin{corollary}\label{cor:2.5}
Let $(X,d)$ be a metric space and let $S,A:X\to X$. Suppose that there exists a $\varphi\in\Phi$ such that the selfmappings $S$ and $A$ satisfy $\varphi$-contractive inequality \eqref{eq:1.13.1}. Suppose also that there exists $(S,A)$-orbit $\mathcal{O}_{SA}(x_{0})$, given by $Ax_{n}=Sx_{n-1}$, $n=1,2,\ldots$. Further assume that
\begin{enumerate}
    \item[(2.5.1)] $X$ is orbitally complete at $x_{0}\in X$;
    \item[(2.5.2)] the pair $(A,S)$ is reciprocal continuous.
\end{enumerate}
If the pair $(A,S)$ is either compatible or compatible of type (A), then $A$ and $S$ have a unique common fixed point.
\end{corollary}

\begin{proof}
Follows from Theorem \ref{thm:2.2} and Theorem \ref{thm:2.3} by taking $S=T$ on $X$.
\end{proof}

We now extend Corollary \ref{cor:2.5} to a sequence of selfmaps.

\begin{corollary}\label{cor:2.6}
Let $(X,d)$ be a metric space and let $\{S_{n}\}_{n=1}^{\infty}$ and $A$ be selfmaps on $X$. Suppose that there exists a $\varphi\in\Phi$ such that the selfmappings $S_{1}$, $S_{j}$ and $A$ satisfy $\varphi$-contractive inequality \eqref{eq:1.13.1}, for each $j=1,2,\ldots$. Suppose also that there exists $(S_{1},A)$-orbit $\mathcal{O}_{S_{1}A}(x_{0})$, given by $Ax_{n}=S_{1}x_{n-1}$, $n=1,2,\ldots$. Further assume that
\begin{enumerate}
    \item[(2.6.1)] $X$ is orbitally complete at $x_{0}\in X$;
    \item[(2.6.2)] the pair $(S_{1},A)$ is reciprocal continuous.
\end{enumerate}
If the pair $(S_{1},A)$ is either compatible or compatible of type (A), then $\{S_{n}\}_{n=1}^{\infty}$ and $A$ have a unique common fixed point.
\end{corollary}

\begin{proof}
By Corollary \ref{cor:2.5}, $S_{1}$ and $A$ have a unique common fixed point $z$ in $X$. Thus
\begin{equation}\label{eq:2.6.3}
S_{1}z=Az=z
\end{equation}

Now, let $j\in \mathbb{N}$ with $j\neq 1$. “Then, using inequality \eqref{eq:1.13.1} and using \eqref{eq:2.6.3},
\end{proof}

we get $d(S_{1}z,S_{j}z)=0$, i.e., $S_{j}z=z$.

The uniqueness of $z$ follows from inequality \eqref{eq:1.13.1}.

Hence, $z$ is the unique common fixed point of the sequence of maps $\{S_{n}\}_{n=1}^{\infty}$ and $A$.

\section{COROLLARIES AND EXAMPLES}

\begin{corollary}\label{cor:3.1}
Let $(X,d)$ be a metric space and let $S$, $T$, $A$ be continuous selfmaps on $X$ satisfying the following condition: assume that there exists $\lambda\in[0,1)$ such that
\begin{equation}\label{eq:3.1.1}
d(Sx,Ty)\leq\lambda\,\max\left\{d(Ax,Ay),\frac{d(Ax,Sx)}{d(Ax,Ay)},\frac{d(Ay,Ty)}{d(Ax,Ay)}\right\}
\end{equation}
with $Ax\neq Ay$ and $d(Sx,Ty)=0$ if $d(Ax,Ay)=0$ for all $x,y\in X$. Suppose that there exists an orbit $\mathcal{O}_{STA}(x_{0})$ of some $x_{0}\in X$ given by \eqref{eq:1.2.1}. Then the sequence $\{Ax_{n}\}$ is a Cauchy sequence.

Further if $X$ is orbitally complete at $x_{0}\in X$ and the pairs of mappings $(A,S)$ and $(A,S)$ are weakly commuting, then $S$, $T$ and $A$ have a unique common fixed point.
\end{corollary}

\begin{proof}
The proof follows from Lemma \ref{lem:2.1} and Theorem \ref{thm:2.2} by choosing $\varphi:\mathbb{R}_{+}\to \mathbb{R}_{+}$ by $\varphi(t)=\lambda t,\lambda\in[0,1),t\geq 0$.
\end{proof}

\begin{remark}\label{rem:3.2}
Theorem \ref{thm:2.2} is a generalization of Corollary \ref{cor:3.1} in view of the following example.
\end{remark}

\begin{example}
Let $X=\mathbb{R}_{+}$ with the usual metric. We define $S,T,A:X\to X$ by $Sx=\frac{x}{1+x}$, $S=T$ and $Ax=x$. We also define $\varphi:\mathbb{R}_{+}\to \mathbb{R}_{+}$ by $\varphi(t)=\frac{t}{1+t}$. We observe that $(S,A)$ is weakly commuting pairs. Now without loss of generality assume that $x>y$. Since $x\neq y$,

\[d(Sx,Ty)=|Sx-Ty|=\frac{x-y}{1+x+y+xy}\text{ and }d(Ax,Ay)=x-y.\]
Hence we get,
\begin{align*}
d(Sx,Ty) &\leq\varphi(d(Ax,Ay))\text{ for all }x,y\in X\text{ and }x\neq y \\
&\leq\varphi\left(\max\left\{d(Ax,Ay),\frac{d(Ax,Sx)}{d(Ax,Ay)},\frac{d(Ay,Ty)}{d(Ax,Ay)},\frac{d(Ax,Ty)}{d(Ax,Ay)}\right\}\right)
\end{align*}
for all $x,y\in X$.

We observe that the mappings $A$, $S$ and $T$ satisfy all the conditions of Theorem \ref{thm:2.2} with the given control $\varphi$. It follows from our theorem that $A$, $S$ and $T$ have a unique common fixed point in $X$, that is $0$.

However the mappings $A$, $S$ and $T$ does not satisfy Corollary \ref{cor:3.1} for otherwise there is a $\lambda\in(0,1)$ such that for all $x\in X$ with $x\neq 0$ and $y=0$
\end{example}

we have $\frac{x}{1+x}=d(Sx,T0)\leq\lambda d(Ax,A0)=\lambda x$ which imply that $\frac{1}{1+x}\leq\lambda$ for any $x\in X$, $x\neq 0$. This is not true. Therefore $A$, $S$ and $T$ do not satisfy Corollary \ref{cor:3.1} for any value of $\lambda\in(0,1)$. Hence, this example shows that Theorem \ref{thm:2.2} is a generalization of Corollary \ref{cor:3.1}.

\begin{example}\label{ex:3.3}
Let $X=(0,1]$ with the usual metric. We define selfmaps $S$, $T$ and $A$ on $X$ by

\[
A(x) = \begin{cases}
1-2x & \text{if } 0<x\leq\frac{1}{3} \\
\frac{1}{6} & \text{if } \frac{1}{3}<x\leq 1,
\end{cases}
\]

\[
Sx = \begin{cases}
x & \text{if } 0<x\leq\frac{1}{3} \\
\frac{1}{3} & \text{if } \frac{1}{3}<x<1 \\
\frac{3}{8} & \text{if } x=1
\end{cases}
\]

and

\[
Tx = \begin{cases}
x & \text{if } 0<x\leq\frac{1}{3} \\
\frac{1}{3} & \text{if } \frac{1}{3}<x<1 \\
\frac{5}{12} & \text{if } x=1
\end{cases}
\]

Here we observe that the pairs of mappings $(A,S)$ and $(A,T)$ are reciprocal continuous and compatible on $X$. Also the selfmaps $S$, $T$ and $A$ satisfy the inequality \eqref{eq:1.13.1} with $\varphi:\mathbb{R}_{+}\to \mathbb{R}_{+}$ defined by $\varphi(t)=\frac{1}{2}t$. Further we observe that when $x_{0}\in(0,\frac{1}{3}]$, $\mathcal{O}_{STA}(x_{0})=\{x_{0},\frac{1}{3},\frac{1}{3},...\}$, when $x_{0}\in(\frac{1}{3},1)$, $\mathcal{O}_{STA}(x_{0})=\{\frac{1}{3},\frac{1}{3},...\}$ and when $x_{0}=1$, $\mathcal{O}_{STA}(x_{0})=\{\frac{3}{8},\frac{5}{16},\frac{1}{3},\frac{1}{3},...\}$ so that $X$ is $(S,T,A)$-orbitally complete at any $x_{0}\in X$. Hence $S$, $T$ and $A$ satisfy all the hypotheses of Theorem \ref{thm:2.2} and $\frac{1}{3}$ is the unique common fixed point of $S$, $T$ and $A$.

We also observe that the pair of selfmaps $(A,S)$ and $(A,T)$ are not compatible of type (A) on $X$; for, take $x_{n}=\frac{1}{3}-\frac{1}{n}$, $n=2,3,4,...$. Then $Ax_{n}=\frac{1}{3}+\frac{2}{n}$, $Sx_{n}=\frac{1}{3}-\frac{1}{n}$ and $Tx_{n}=\frac{1}{3}-\frac{1}{n}$, $n=2,3,4,...$. We note that $Ax_{n}\rightarrow\frac{1}{3}$ as $n\rightarrow\infty$, $Sx_{n}\rightarrow\frac{1}{3}$ as $n\rightarrow\infty$ and $Tx_{n}\rightarrow\frac{1}{3}$ as $n\rightarrow\infty$. Now $A^{2}x_{n}=\frac{1}{6}$, $SAx_{n}=\frac{1}{3}$ and $TAx_{n}=\frac{1}{3}$ for all $n=2,3,4,...$ so that $\lim\limits_{n\rightarrow\infty}d(A^{2}x_{n},SAx_{n})=\frac{1}{6}\neq 0$ and $\lim\limits_{n\rightarrow\infty}d(A^{2}x_{n},TAx_{n})=\frac{1}{6}\neq 0$.
\end{example}

\begin{example}\label{ex:3.4}
Let $X=(0,2]$ with the usual metric. We define mappings $A,S,T:X\to X$ by

\[
Ax = \begin{cases}
2 & \text{if } 0<x\leq\frac{1}{2} \\
2x-1 & \text{if } \frac{1}{2}<x\leq 1 \\
2 & \text{if } 1<x\leq 2,
\end{cases}
\]

\[
Sx = \begin{cases}
\frac{3}{2} & \text{if } 0<x\leq\frac{1}{2} \\
x & \text{if } \frac{1}{2}<x\leq 1 \\
\frac{3}{2} & \text{if } 1<x\leq 2
\end{cases}
\]

and

\[
Tx = \begin{cases}
\frac{5}{4} & \text{if } 0<x\leq\frac{1}{2} \\
x & \text{if } \frac{1}{2}<x\leq 1 \\
\frac{5}{4} & \text{if } 1<x\leq 2
\end{cases}
\]

Here we observe that the pairs of mappings $(A,S)$ and $(A,T)$ are reciprocally continuous, compatible and compatible of type (A) on $X$. Also the selfmaps $S$, $T$ and $A$ satisfy the inequality \eqref{eq:1.13.1} with $\varphi:\mathbb{R}_{+}\to \mathbb{R}_{+}$ defined by $\varphi(t)=\frac{1}{2}t$. Further we observe that when $x_{0}\in(\frac{1}{2},1)$, $\mathcal{O}_{STA}(x_{0})=\{1-\frac{1}{2^{n-1}}+\frac{1}{2^{n-1}}x_{0}\}_{n=1}^{\infty}$ which shows that $Ax_{n}\neq Ax_{n+1}$ for all $n$ and $Ax_{n}\to 1$ as $n\rightarrow\infty$ and $1$ is not in $\mathcal{O}_{STA}(x_{0})$ so that $X$ is $(S,T,A)$-orbitally complete at any $x_{0}\in(\frac{1}{2},1)$. Of course $1\in\mathcal{O}_{STA}(x_{0})$. Hence $S$, $T$ and $A$ satisfy all the hypotheses of Theorem \ref{thm:2.3} and 1 is the unique common fixed point of $S$, $T$ and $A$.
\end{example}

\begin{example}\label{ex:3.5}
Let $X=(0,1]$ with the usual metric. We define selfmaps $S_{n}$ and $A$ on $X$ by

\[
A(x) = \begin{cases}
1-2x & \text{if } 0<x\leq\frac{1}{3} \\
\frac{1}{6} & \text{if } \frac{1}{3}<x\leq 1,
\end{cases}
\]

and
\end{example}

\[
S_{n}x = \begin{cases}
x & \text{if } 0<x\leq\frac{1}{3} \\
\frac{1}{3} & \text{if } \frac{1}{3}<x<1 \\
\frac{5}{12}-\frac{1}{24n} & \text{if } x=1
\end{cases}
\]

for each $n=1,2,\ldots$.

Here we observe that, for each $n$, the pair of mappings $(S_{n},A)$ is reciprocal continuous and compatible on $X$. Also, for each $n$, the selfmaps $S_{n}$ and $A$ satisfy the inequality \eqref{eq:1.13.1} with $\varphi:\mathbb{R}_{+}\to \mathbb{R}_{+}$ defined by $\varphi(t)=\frac{1}{2}t$.

Further we observe that when $x_{0}\in(0,\frac{1}{3}]$, $\mathcal{O}_{S_{1}A}(x_{0})=\{x_{0},\frac{1}{3},\frac{1}{3},\ldots\}$, when $x_{0}\in(\frac{1}{3},1)$, $\mathcal{O}_{S_{1}A}(x_{0})=\{\frac{1}{3},\frac{1}{3},\ldots\}$ and when $x_{0}=1$, $\mathcal{O}_{S_{1}A}(x_{0})=\{\frac{3}{8},\frac{5}{16},\frac{1}{3},\frac{1}{3},\ldots\}$ so that $X$ is $(A,S)$-orbitally complete at any $x_{0}\in X$. Hence $S$, $T$ and $A$ satisfy all the hypotheses of Corollary \ref{cor:2.6} and $\frac{1}{3}$ is the unique common fixed point of $S_{n}$ and $A$, $n=1,2,\ldots$.

In Example \ref{ex:3.3}, we observe that $S_{1}$ and $A$ are not compatible of type (A) mappings.

In the following few examples we show if any condition of the hypotheses of Theorem \ref{thm:2.2} and Theorem \ref{thm:2.3} fails to hold then $S$, $T$ and $A$ may not have a common fixed point.

\begin{example}\label{ex:3.6}
Let $X=(0,1]$ with the usual metric. We define selfmaps $S$, $T$ and $A$ on $X$ by

\[
A(x) = \begin{cases}
1-2x & \text{if } 0<x<\frac{1}{3} \\
\frac{1}{6} & \text{if } \frac{1}{3}\leq x\leq 1,
\end{cases}
\]

\[
S(x) = \begin{cases}
x & \text{if } 0<x\leq\frac{1}{3} \\
\frac{1}{3} & \text{if } \frac{1}{3}<x\leq 1,
\end{cases}
\]

and $T=S$.

We observe that the pair of mappings $(S,A)$ is compatible on $X$. Also the selfmaps $S$, $T$ and $A$ satisfy the inequality \eqref{eq:1.13.1} with $\varphi:\mathbb{R}_{+}\to \mathbb{R}_{+}$ defined by $\varphi(t)=\frac{1}{2}t$. But the pair of selfmaps $(S,A)$ is not reciprocal continuous on $X$; for if $x_{n}=\frac{1}{3}-\frac{1}{n}$, $n=2,3,4,\ldots$, then $Ax_{n}=\frac{1}{3}+\frac{2}{n}$ and $Sx_{n}=\frac{1}{3}-\frac{1}{n}$. We note that $Ax_{n}\rightarrow\frac{1}{3}$ as $n\rightarrow\infty$ and $Sx_{n}\rightarrow\frac{1}{3}$ as $n\rightarrow\infty$. Now $ASx_{n}=\frac{1}{3}+\frac{2}{n}$ $n=2,3,4,\ldots$ so that $\lim\limits_{n\rightarrow\infty}ASx_{n}=\frac{1}{3}\neq\frac{1}{6}=A(\frac{1}{3})$
\end{example}

We observe that $A$, $S$ and $T$ have no common fixed point.

\begin{example}\label{ex:3.7}
Let $X=[0,\frac{1}{2}]$ with the usual metric. We define selfmaps $S$, $T$ and $A$ on $X$ by

\[
A(x) = \begin{cases}
0 & \text{if } 0\leq x\leq\frac{1}{3} \\
x & \text{if } \frac{1}{3}<x\leq\frac{1}{2},
\end{cases}
\]

\[
S(x) = \begin{cases}
\frac{1}{2} & \text{if } 0\leq x<\frac{1}{3} \\
\frac{1}{2}-\frac{1}{2}x & \text{if } \frac{1}{3}\leq x\leq\frac{1}{2},
\end{cases}
\]

and $T=S$.

We observe that the pair of mappings $(S,A)$ is reciprocal continuous on $X$. Also the selfmaps $S$, $T$ and $A$ satisfy the inequality (1.13.1) with $\varphi:\mathbb{R}_{+}\to \mathbb{R}_{+}$ defined by $\varphi(t)=\frac{1}{2}t$. While the selfmaps $S$ and $A$ are neither compatible nor compatible of type (A) on $X$; for if $x_{n}=\frac{1}{3}+\frac{1}{n}$, $n=4,5,6,\ldots$, then we get $Ax_{n}=\frac{1}{3}+\frac{1}{n}$ and $Sx_{n}=\frac{1}{3}-\frac{1}{n}$. We note that $Ax_{n}\rightarrow\frac{1}{3}$ as $n\rightarrow\infty$ and $Sx_{n}\rightarrow\frac{1}{3}$ as $n\rightarrow\infty$. Now $SAx_{n}=\frac{1}{3}-\frac{1}{2n}$, $ASx_{n}=0$ and $S^{2}x_{n}=\frac{1}{2}$ so that $\lim\limits_{n\rightarrow\infty}d(S^{2}x_{n},ASx_{n})=\frac{1}{2}\neq 0$ and $\lim\limits_{n\rightarrow\infty}d(SAx_{n},ASx_{n})=\frac{1}{3}\neq 0$.

We observe that $A$, $S$ and $T$ have no common fixed point.
\end{example}

\begin{note}\label{note:3.8}
If we replace the inequality \eqref{eq:1.13.1} with
\begin{equation}\label{eq:1.13.1'}
d(Sx,Ty)d(Ax,Ay)\leq\varphi\left(\max\left\{d(Ax,Ay)^{2},d(Ax,Sx) d(Ay,Ty), d(Ax,Ty) d(Ay,Sx)\right\}\right)
\end{equation}
for all $x,y\in X$ and $\varphi\in\Phi$ then also Lemma \ref{lem:2.1}is valid without the condition $Ax\neq Ay$ for all $x,y$.
\end{note}

\begin{note}\label{note:3.9}
Theorem \ref{thm:1.1} and Theorem \ref{thm:1.5} remain valid and are corollary to Theorem \ref{thm:2.2} and Theorem \ref{thm:2.3} by taking $\varphi(t)=\lambda t$, where $\lambda=\alpha+\max\{\beta,\gamma\}$.
\end{note}

\section{Convergence of fixed points for $\varphi$-contractive inequality involving rational expressions}

\begin{theorem}\label{thm:4.1}
Let $\{S_{n}\}_{n=1}^{\infty}$ and $A$ be selfmaps on a metric space $(X,d)$ with $u_{n}$ as a common fixed points of $S_{n}$ and $A$ for $n=1,2,3,\ldots$ and $A$ be continuous on $X$. Suppose that there exists a $\varphi\in\Phi$ such that the selfmappings $S_{n}$ and $A$ satisfy $\varphi$-contractive inequality \eqref{eq:1.13.2}, for each $n=1,2,\ldots$
\end{theorem}

If $\{S_{n}\}_{n=1}^{\infty}$ converges pointwise to $S$, then $u_{n}\to u$ as $n\to\infty$ if and only if $Su=Au=u$. Further, $u$ is the unique common fixed point of $A$ and $S$.

\begin{proof}
Suppose that $u_{n}\to u$ as $n\to\infty$. Then, since $A$ is continuous on $X$, $Au_{n}\to Au$ as $n\to\infty$. Assume that $Au_{n}\neq Au$ for all $n$.

Now
\begin{align*}
d(u,Au) &\leq d(u,Au_{n})+d(Au_{n},Au) \\
&= d(u,u_{n})+d(Au_{n},Au)
\end{align*}

Letting $n\to\infty$, we get
\begin{equation}\label{eq:4.1.1}
d(u,Au)=0\text{, i.e., }Au=u.
\end{equation}

Now consider
\begin{align*}
d(u,Su) &\leq d(u,u_{n})+d(S_{n}u_{n},S_{n}u)+d(S_{n}u,Su) \\
&\leq d(u,u_{n})+\varphi\left(\max\left\{d(Au_{n},Au),\frac{d(Au_{n},S_{n}u_{n})d(Au,S_{n}u)}{d(Au_{n},Au)},\right.\right. \\
&\quad \left.\left.\frac{d(Au_{n},S_{n}u)d(Au,S_{n}u_{n})}{d(Au_{n},Au)}\right\}\right)+d(S_{n}u,Su) \\
&= d(u,u_{n})+\varphi\left(\max\left\{d(u_{n},u),0,d(u_{n},S_{n}u)\right\}\right)+d(S_{n}u,Su) \\
&\leq d(u,u_{n})+\varphi\left(\max\left\{d(u_{n},u),d(u_{n},u)+d(u,Su)+d(Su,S_{n}u)\right\}\right) \\
&\quad +d(S_{n}u,Su)
\end{align*}

Letting $n\to\infty$, using the continuity of $\varphi$ we get
\[
d(u,Su)\leq\varphi(d(u,Su)).
\]

Hence,
\begin{equation}\label{eq:4.1.2}
d(u,Su)=0\text{, i.e., }Su=u.
\end{equation}

Therefore, from \eqref{eq:4.1.1} and \eqref{eq:4.1.2} we get
\[
Au=Su=u
\]

Conversely, suppose that $Au=Su=u$.

We claim that $u_{n}\to u$ as $n\to\infty$.

If $u_{n}=u$ for large $n$'s, then $d(Au_{n},Au)=0$ for that $n$'s.

Hence $d(S_{n}u_{n},S_{n}u)=0$ for that $n$'s. But $u_{n}=S_{n}u_{n}$. Hence $u_{n}=S_{n}u$ for that $n$'s.

Hence, $\lim_{n\to\infty}u_{n}=\lim_{n\to\infty}S_{n}u=u$.

Therefore, $u_{n}\to u$ as $n\to\infty$.

Now assume that $u_{n}\neq u$ for all $n$.

Consider
\begin{align*}
d(u_{n},u) &= d(S_{n}u_{n},Su) \\
&\leq d(S_{n}u_{n},S_{n}u)+d(S_{n}u,Su) \\
&\leq \varphi\left(\max\left\{d(Au_{n},Au),\frac{d(Au_{n},S_{n}u_{n})d(Au,S_{n}u)}{d(Au_{n},Au)},\right.\right. \\
&\quad \left.\left.\frac{d(Au_{n},S_{n}u)d(Au,S_{n}u_{n})}{d(Au_{n},Au)}\right\}\right)+d(S_{n}u,Su) \\
&= \varphi\left(\max\left\{d(u_{n},u),0,d(u_{n},S_{n}u)\right\}\right)+d(S_{n}u,Su) \\
&\leq \varphi\left(\max\left\{d(u_{n},u),d(u_{n},u)+d(u,Su)+d(Su,S_{n}u)\right\}\right) \\
&\quad +d(S_{n}u,Su) \\
&= \varphi\left(\max\left\{d(u_{n},u),d(u_{n},u)+d(Su,S_{n}u)\right\}\right) \\
&\quad + d(S_{n}u,Su)
\end{align*}

Now
\begin{align*}
\lim_{n\to\infty}d(u_{n},u) &\leq \lim_{n\to\infty}[\varphi\left(\max\left\{d(u_{n},u),d(u_{n},u)+d(Su,S_{n}u)\right\}\right) \\
&\quad + d(S_{n}u,Su)]
\end{align*}

Hence, $\lim_{n\to\infty}d(u_{n},u)=0$, i.e., $u_{n}\to u$ as $n\to\infty$.

For the uniqueness part we proceed as follows.

We claim that $u_{n}$ is the unique common fixed point of $S_{n}$ and $A$, for each $n=1,2,\ldots$.

Suppose that there exists $z_{n}\in X$ such that $S_{n}z_{n}=Az_{n}=z_{n}$ and $z_{n}\neq u_{n}$ for some $n$. Then
\begin{align*}
d(z_{n},u_{n}) &= d(S_{n}z_{n},S_{n}u_{n}) \\
&\leq \varphi\left(\max\left\{d(Az_{n},Au_{n}),\frac{d(Az_{n},S_{n}z_{n})d(Au_{n},S_{n}u_{n})}{d(Az_{n},Au_{n})},\right.\right. \\
&\quad \left.\left.\frac{d(Az_{n},S_{n}u_{n})d(Au_{n},S_{n}z_{n})}{d(Az_{n},Au_{n})}\right\}\right) \\
&= \varphi(d(z_{n},u_{n})),\text{ a contradiction.}
\end{align*}

Hence, $d(z_{n},u_{n})=0$, i.e., $z_{n}=u_{n}$. Hence the claim holds.

Finally, suppose that there exists $z\in X$ with $z\neq u$ and $Az=Sz=z$. Then by Theorem \ref{thm:4.1}, $u_{n}\to z$ as $n\to\infty$. 

But $u_{n}\to u$ as $n\to\infty$ and $X$ is a metric space which is a Hausdorff Space.

Hence, $u=z$.

This completes the proof of the theorem.
\end{proof}
\section{Conclusion}

In this paper, we have established a unified framework for proving the existence and uniqueness of common fixed points for three selfmaps defined on orbitally complete metric spaces. Our approach hinges on a generalized $\varphi$-contractive condition involving rational expressions, which extends classical contraction principles and accommodates more flexible mappings. 

By relaxing standard completeness assumptions to orbital completeness and replacing continuity with reciprocal continuity and compatibility (or compatibility of type (A)), we have broadened the applicability of fixed point theory. Our results subsume and generalize earlier theorems by Jaggi \cite{Jaggi1975} and Phaneendra et al.\cite{Phaneendra2007}, while also providing new insights into the structure of fixed point iterations. 

We have further demonstrated the robustness of our framework by presenting illustrative examples and convergence results for sequences of selfmaps. These contributions set the stage for future extensions to more abstract settings such as partial metric spaces, cone metric spaces, and applications in iterative approximation and nonlinear analysis.

\section*{ACKNOWLEDGMENTS}
The authors sincerely thank Prof. K. P. R. Sastry for his initiation of this paper.

\end{document}